\documentclass[a4paper,11pt]{article}
\usepackage{amssymb,amsthm,amsmath,latexsym,url}

\newtheorem{thm}{Theorem}[section]

\newtheorem{lem}[thm]{Lemma}

\newtheorem*{question}{Question}

\newcommand{\T}{\mathcal{T}}

\DeclareMathOperator{\Aut}{Aut}
\DeclareMathOperator{\St}{st}

\date{}

\title{A note on Engel elements in the\\ first Grigorchuk group}

\author{\small{\textsc{Marialaura Noce}\footnote{The authors are members of {\em National Group for Algebraic and Geometric Structures, and their Applications} (GNSAGA--INdAM).}}\\
\small{Dipartimento di Matematica, Universit\`a di Salerno, Italy}\\
\small{Matematika Saila, Euskal Herriko Unibertsitatea UPV/EHU, Spain}\\
\small{E-mail: mnoce@unisa.it}\\
[10pt]
\small{\textsc{Antonio Tortora}\footnotemark[1]}\\
\small{Dipartimento di Matematica, Universit\`a di Salerno, Italy}\\
\small{E-mail: antortora@unisa.it}}
\begin{document}
\maketitle

\begin{abstract}
Let $\Gamma$ be the first Grigorchuk group. According to a result of Bar\-thol\-di, the only left Engel elements of $\Gamma$ are the involutions. This implies that the set of left Engel elements of $\Gamma$ is not a subgroup. Of particular interest is to wonder whether this happens also for the sets of bounded left Engel elements, right Engel elements, and bounded right Engel elements of $\Gamma$. Motivated by this, we prove that these three subsets of $\Gamma$ coincide with the identity subgroup.\\

\end{abstract}

\section{Introduction}
Let $G$ be a group. An element $g\in G$ is called a left Engel element if for any $x\in G$ there exists a positive integer $n=n(g,x)$ such that $[x,_n g]=1$. As usual, the commutator $[x,_n g]$ is defined inductively by the rules
$$[x,_1 g]=[x,g]=x^{-1}x^g\quad {\rm and,\, for}\; n\geq 2,\quad [x,_n g]=[[x,_{n-1} g],g].$$ If $n$ can be chosen independently of $x$, then $g$ is called a left $n$-Engel element, or less precisely a bounded left Engel element. Similarly, $g$ is a right Engel element or a bounded right Engel element if the variable $x$ appears on the right. The group $G$ is then called Engel (or bounded Engel, resp.) if all its elements are both left and right Engel (or bounded Engel, resp.). We denote by $L(G),\overline{L}(G),R(G)$ and $\overline{R}(G)$ respectively the sets of left Engel elements, bounded left Engel elements, right Engel elements, and bounded right Engel elements of $G$. It is clear that these four subsets are invariant under
automorphisms of $G$. Furthermore, by a well-known result of Heineken (see \cite[12.3.1]{Rob}), we have 
\begin{equation}\label{Heineken}
R(G)^{-1}\subseteq L(G)\quad {\rm and}\quad \overline{R}(G)^{-1}\subseteq \overline{L}(G). \tag{*}
\end{equation}

It is a very long-standing question whether the sets $L(G),\overline{L}(G), R(G)$ and $\overline{R}(G)$ are subgroups of $G$ (see Problems 16.15 and 16.16 in \cite{MK}). There are several classes of groups for which this is true (see \cite{Abd} for an account; see also \cite{BMTT, STT}). On the other hand the question is still open, except for $L(G)$ when $G$ is a 2-group. In fact, in this case, one can easily see that any involution is a left Engel element \cite[Proposition 3.3]{Abd}. However, according to an example of Bludov, there exists a 2-group generated by involutions with an element of order four which is not left Engel (\cite{Blu}, see \cite{ML} for a proof). This suggests the following question.

\begin{question}[Bludov]
Assuming that $G$ is not a $2$-group, is $L(G)$ a subgroup of $G$? 
\end{question}

We point out that the group $G$ considered by Bludov is based on the (first) Grigorchuk group \cite{Gri}, that we denote throughout by $\Gamma$. More precisely, $G$ is the wreath product $D_8\ltimes \Gamma^4$ where $D_8$ is the dihedral group of order 8. Since  $\Gamma$ is a 2-group generated by involutions, one might wonder whether $\Gamma$ is an Engel group but the answer is negative, as shown by Bartholdi: 

\begin{thm}[\cite{Ba016}, see also \cite{Ba017}]\label{Bartholdi}
Let $\Gamma$ be the first Grigorchuk group. Then 
$$L(\Gamma) = \{g \in \Gamma \mid g^{2} = 1\}.$$
In particular, $\Gamma$ is not an Engel group.
\end{thm}

The question now arises: {\em are $\bar{L}(\Gamma),R(\Gamma)$ and $\bar{R}(\Gamma)$ subgroups of $\Gamma$?}\\ Recall that $\Gamma$ is just-infinite, that is, $\Gamma$ is an infinite group all of whose proper quotients are finite. As a consequence, if $\overline{L}(\Gamma)$ were a (proper) subgroup of $\Gamma$, then $\overline{L}(\Gamma)$ would be finitely generated and, by Theorem \ref{Bartholdi}, also abelian. Hence $\overline{L}(\Gamma)$ should be finite and therefore trivial, being $\Gamma$ an infinite 2-group and $\overline{L}(\Gamma)$ of finite index.  
Notice also that, by (\ref{Heineken}), the same holds for $R(\Gamma)$ and $\overline{R}(\Gamma)$. 

Motivated by this, in the present note we prove the following theorem.  
	
\begin{thm}\label{main}
Let $\Gamma$ be the first Grigorchuk group. Then 
$$\bar{L}(\Gamma)=R(\Gamma)=\bar{R}(\Gamma)=\{1\}.$$
\end{thm}

The proof of Theorem \ref{main} will be given in the next section. 

\section{The proof}

Before proving Theorem \ref{main}, we recall how the Grigorchuk group $\Gamma$ is defined. We also collect some properties of $\Gamma$ on which depends our proof. For a more detailed account on $\Gamma$, we refer to \cite[Chapter 8]{dlH}.

Let $\mathcal{T}$ be the regular binary rooted tree with vertices indexed by $X^*$,
the free monoid on the alphabet $X=\{0,1\}$. An automorphism of $\mathcal{T}$ is a bijection of the vertices that preserves incidence. The set $\Aut\T$ of all automorphisms of $\mathcal{T}$ is a group with respect to composition. The stabilizer $\St(n)$ of the $n$th level of $\mathcal{T}$ is the normal subgroup of $\Aut\T$ consisting of the automorphisms leaving fixed all words of length $n$.

If an automorphism $g$ fixes a vertex, then the restriction $g_i$ of $g$ to the subtree hanging from this vertex induces an automorphism of $\T$. In particular, if $g\in\St(n)$ then $g_i$ is defined for $i=1,\ldots,2^n$, and one can consider the injective homomorphism
$$\psi_n:g\in \St(1)\longmapsto (g_1,\ldots, g_{2^n})\in  \Aut\T\times\overset{2^n}{\ldots}\times\Aut\T.$$
We write $\psi$ instead of $\psi_1$. Assuming $\psi(g)=(g_1,g_2)$, it is easy to see that
$$\psi(g^a)=(g_2,g_1),$$
where $a$ is the rooted automorphism of $\T$ corresponding to the permutation $(0\,1)$; this will be used freely in the sequel.
 
The Grigorchuk group $\Gamma$ is the subgroup of $\Aut\T$ generated by the rooted automorphism $a$, and the automorphisms $b,c,d\in \St(1)$ which are defined recursively as follows:
$$\psi(b)=(a,c),\ \psi(c)=(a,d),\ \psi(d)=(1,b).$$
Moreover,
$$\Gamma=\langle a \rangle \ltimes \St_{\Gamma}(1)$$
where $\St_{\Gamma}(1)=\Gamma \cap \St(1)$. Recall also that $\Gamma$ is spherically transitive (i.e., it acts transitively on each level of $\T$) and it has a subgroup $K$ of finite index such that $\psi(K)\supseteq K\times K$. In other words, $\Gamma$ is regular branch over $K$.

For the proof of Theorem \ref{main} we require two lemmas concerning commutators between specific elements of $\Gamma$. 

\begin{lem}\label{induction}
Let $x\in\Gamma$ be an involution and $y\in \St_{\Gamma}(1)$, with $\psi(y)=(k,1)$. Suppose $x=ag$, where $g \in \St_{\Gamma}(1)$ and $\psi(g)=(g_{1}, g_{2})$. Then, for every $m\geq 1$, we have
$$\psi([y,_m x]) = (k^{(-1)^{m}\,2^{m-1}}, (k^{g_2})^{(-1)^{m-1}\,2^{m-1}}).$$
\end{lem}

\begin{proof}
Since $a^2=1$, we have $(ag)^{2}=1$ and so $g^a=g^{-1}$. Thus
$$(g_2,g_1)=\psi(g^{a})=\psi(g^{-1})=(g_{1}^{-1}, g_{2}^{-1}),$$
from which it follows that $g_{2}=g_{1}^{-1}$. 

We now proceed by induction on $m$. The case $m=1$ is clear, indeed  
$$\psi([y,ag])=\psi(y^{-1}y^{ag})=(k^{-1},1)(1,k^{g_2})=(k^{-1}, k^{g_2}).$$ 
Let $m>1$. Then, by using the induction hypothesis and that $g_2g_1=1$, we get
\begin{align*}
\psi([y,_m ag])& = \psi([y,_{m-1} ag]^{-1} [y,_{m-1} ag]^{ag}) \\
& = \left(k^{(-1)^m\,2^{m-2}}, (k^{g_2})^{(-1)^{m-1}\,2^{m-2}}\right)\left((k^{g_2g_1})^{(-1)^{m-2}\,2^{m-2}}, (k^{g_2})^{(-1)^{m-1}\,2^{m-2}}\right)\\
& = \left(k^{(-1)^{m}\,2^{m-1}}, (k^{g_2})^{(-1)^{m-1}\,2^{m-1}}\right), 
\end{align*}
as desired.
\end{proof}

\begin{lem}\label{induction2}
Let $x\in\Gamma$ and $y\in\St_{\Gamma}(1)$, with $\psi(y)=(y_1,y_2)$. Suppose $x=ag$, where $g \in \St_{\Gamma}(1)$ and $\psi(g)=(g_{1}, g_{2})$. Then, for every $m\geq 1$, we have
$$\psi([x,_{m+1} y]) = ([(y_{2}^{-1})^{g_{1}},_{m} y_{1}]^{y_{1}}, [(y_{1}^{-1})^{g_{2}},_{m} y_{2}]^{y_{2}}).$$
\end{lem}

\begin{proof}
Of course, $[x,_{n} y] \in \St_{\Gamma}(1)$ for every $n\geq 1$. Thus
\begin{align*}
\psi([x, y]) & = \psi((y^{-1})^{x}y) = \psi((y^{-1})^{a})^{\psi(g)}\psi(y) = ((y_{2}^{-1})^{g_{1}}, (y_{1}^{-1})^{g_{2}})(y_{1}, y_{2}) \\
& = ((y_{2}^{-1})^{g_{1}}y_{1}, (y_{1}^{-1})^{g_{2}}y_{2}).
\end{align*}
It follows that
\begin{align*}
\psi([x, y, y]) & = [\psi([x, y]), \psi(y)] \\
& = [((y_{2}^{-1})^{g_{1}}y_{1}, (y_{1}^{-1})^{g_{2}}y_{2}),(y_1,y_2)] \\
& = ([(y_{2}^{-1})^{g_{1}}y_{1}, y_{1}], [(y_{1}^{-1})^{g_{2}}y_{2}, y_{2}]) \\
& = ([(y_{2}^{-1})^{g_{1}}, y_{1}]^{y_{1}}, [(y_{1}^{-1})^{g_{2}}, y_{2}]^{y_{2}}).
\end{align*}
This proves the result when $m=1$. Let $m>1$. Then, by induction, we conclude that
\begin{align*}
\psi([x, _{m+1} y]) & = [\psi([x,_{m} y]), \psi(y)] \\
& = [([(y_{2}^{-1})^{g_{1}},_{m-1} y_{1}]^{y_{1}}, [(y_{1}^{-1})^{g_{2}},_{m-1} y_{2}]^{y_{2}}), (y_1,y_2)] \\
& = ([(y_{2}^{-1})^{g_{1}},_{m} y_{1}]^{y_{1}}, [(y_{1}^{-1})^{g_{2}},_{m} y_{2}]^{y_{2}}).
\end{align*}

\end{proof}

We are now ready to prove Theorem \ref{main}. 

\begin{proof}[\bf Proof of Theorem \ref{main}] 
Let $x$ be a nontrivial Engel element of $\Gamma$. First, notice that we may assume $x\notin \St_{\Gamma}(1)$. In fact, if $x\in \St_{\Gamma}(n)\backslash \St_{\Gamma}(n+1)$ then $$\psi_{n}(x) = (x_{1}, \dots, x_{2^{n}})$$
where all the $x_i$'s are Engel elements of the same kind of $x$ and one of $x_i$'s does not belong to $\St_{\Gamma}(1)$. Hence $x=ag$, for some $g\in \St_{\Gamma}(1)$ with $\psi(g)=(g_1,g_2)$. We distinguish two cases: $x\in \overline{L}(\Gamma)$ and $x\in R(\Gamma)$. 	

Assume $x\in\overline{L}(\Gamma)$. Then $[y,_n x]=1$ for some $n\geq 1$ and for every $y\in \Gamma$. Also $x^2=1$, by Theorem \ref{Bartholdi}. Since $K$ is not of finite exponent, we can take $k\in K$ of order $>2^{m-1}$ for some $m$. On the other hand $\psi(K)\supseteq K\times K$, so there exists $y\in K\subseteq\St_{\Gamma}(1)$ such that $(k,1)=\psi(y)$. Thus, by Lemma \ref{induction}, we have
$$(1,1)=\psi(1)=\psi([y,_m x])=\left(k^{(-1)^{m}\,2^{m-1}}, (k^{g_2})^{(-1)^{m-1}\,2^{m-1}}\right).$$
It follows that $k^{2^{m-1}}=1$, a contradiction. This proves that $\overline{L}(\Gamma) = \{1\}$. 

Assume $x\in R(\Gamma)$. Since $K$ is not abelian, it cannot be an Engel group by Theorem \ref{Bartholdi}. Thus $[h,_{n} y_1]\neq 1$
for some $h,y_1\in K$ and for every $n\geq 1$. Put $y_2=[y_1, h]^{g_1^{-1}}$. Obviously, $y_2\in K$ and $(y_2^{-1})^{g_1}=[h,y_1]$. Now $\Gamma$ is regular branch over $K$, so there exists $y\in K\subseteq\St_{\Gamma}(1)$ such that $\psi(y)=(y_1,y_2)$. Furthermore, there is $m=m(x,y)\geq 1$ such that $[x,_m y]=1$. Applying Lemma \ref{induction2}, we get
\begin{align*}
(1,1) & =\psi(1) =\psi([x,_{m+1} y])  \\
& =([(y_{2}^{-1})^{g_{1}},_{m} y_{1}]^{y_{1}}, [(y_{1}^{-1})^{g_{2}},_{m} y_{2}]^{y_{2}}) \\
& = ([h,_{m+1} y_1]^{y_1}, [(y_{1}^{-1})^{g_{2}},_{m} y_{2}]^{y_{2}}).
\end{align*}
This implies that $[h,_{m+1} y_1]=1$, which is a contradiction. Therefore $R(\Gamma)=\overline{R}(\Gamma) = \{1\}$, and the proof of Theorem \ref{main} is complete.
\end{proof}

\vspace {0.5cm}

\noindent{\bf Aknowledgements.} The authors would like to thank Prof. Gustavo A. Fern\'andez-Alcober for interesting and helpful conversations (at {\em cafeter\'ia}).

\end{document}